\documentclass{amsart}

\usepackage{amsthm,amssymb,verbatim,color}
\usepackage{graphicx}
\usepackage{enumerate}
\usepackage{enumitem}
\usepackage{hyperref}

\usepackage[alphabetic, msc-links]{amsrefs}

\usepackage[parfill]{parskip}
\usepackage[utf8]{inputenc}
\usepackage{xcolor}
\usepackage{amsmath,amssymb,amsfonts,amsthm}

\newtheorem{theorem}{Theorem}[section]
\newtheorem*{theorem*}{Theorem}

\newtheorem{lemma}[theorem]{Lemma}

\theoremstyle{definition}

\newtheorem{remark}[theorem]{Remark}

\newcommand{\cA}{\mathcal{A}}

\newcommand{\ddt}{\frac{d}{dt}}

\newcommand{\norm}[1]{\left\lvert{#1}\right\rvert}

\begin{document}

\title{How close is too close for singular mean curvature flows?}

\author{J.M. Daniels-Holgate}

\author{Or Hershkovits}

\date{\today}

\maketitle

\begin{abstract}
Suppose $(M^i_t)_{t\in [0,T)}$, $i=1,2$, are two mean curvature flows in $\mathbb{R}^{n+1}$ encountering a multiplicity one compact singularity at time $T$, in such a manner that for every $k$, the Hausdorff distance between the two flows, $d_H$, satisfies  $d_{H}(M^1_t,M^2_t)/(T-t)^k \rightarrow 0$. We demonstrate that $M^1_t=M^2_t$ for every $t$. This generalizes a result of Martin-Hagemayer and Sesum, who proved the case where $M^1_t$ is itself a self-similarly shrinking flow.   
\end{abstract}

\section{Introduction}
Singularity formation is a central area of study in Geometric Analysis. An important consequence of Huisken's monotoncity formula for mean curvature flow (MCF) is that a tangent flow at a singular point in an MCF is given by a self-shrinking mean curvature flow,
\begin{align*}
    \Sigma_t = \sqrt{-t}\Sigma_{-1}, t\in (-\infty, 0).
\end{align*} 

It was shown by Schulze, \cite{Schulze_cpt}, that when a tangent flow is given by a  multiplicity one compact self-shrinker, the tangent flow is in fact unique.  Entwined with the question of uniqueness of tangents is the question of minimal rate of convergence to the tangent (cf.\cite{AA,Simon,CM_loj}). More generally, one can ask \textit{at which rates can the flow converge to the self-shrinker?}

{ For the case of spherical singularities, the minimal and maximal rates of convergence were obtained by Sesum  \cite{Sesum} and Strehlke \cite{Stre}, the latter showing that a solution to the MCF converging super polynomially to the family of shrinking spheres must itself be a family of shrinking spheres. Strehlke's result was recently generalized by  Martin-Hagemayer and Sesum,\cite{MartinHagemayerSesum} (following a similar result for Ricci flow by Kotschwar \cite{Ks}) to apply to a flow converging super polynomially to any self shrinking \textit{compact} singularity. Importantly, obtaining a polynomial bound on the decay allows one to perform secondary blow-up arguments resulting in a non zero solution to the linearized equation. 
}

{ Note that if $M^1_t$ and $M^2_t$ are two distinct MCFs with the same compact tangent, \cite{MartinHagemayerSesum,Stre} give no lower bound as to how close they might be, unless one of them is the shrinking flow  itself. The purpose of this note is to provide such sharp lower bound.
}
\begin{theorem}\label{main_theorem}
    Let $(M^i_t)_{t\in [0,T)}$  ($i=1,2$) be two compact mean curvature flows in $\mathbb{R}^{n+1}$, encountering a multiplicity one compact singularity at $(x_0,T)$, such that for every $k\in \mathbb{N}$
    \begin{equation}\label{dist_ass}
    \lim_{t\rightarrow T}\frac{d_{H}\left(M^1_t,M^2_t\right)}{(T-t)^k}=0.
    \end{equation}  
    Then $M^1_t=M^2_t$ for every $t\in [0,T)$.
    \end{theorem}

{
Note that one recovers the results of Martin-Hagemayer--Sesum by setting $M^1_t=\sqrt{-t}\Sigma$ for a compact self-shrinker $\Sigma$. In essence, \cite{MartinHagemayerSesum} shows a flow is linearly distinguishable from its tangent (as long as it's not self shrinking), whilst Theorem \ref{main_theorem} means that any two such distinct flows are linearly distinguishable from one another.}

{
While both our proof and \cite{MartinHagemayerSesum} are based on some ``frequency function'' argument, our proof is quite different from the one therein, as we now explain.

\bigskip

Considering the rescaled hypersurfaces
\begin{equation}
N^i_\tau:=e^{\tau/2}\left(M^{i}_{T-e^{-\tau}}-x_0\right),
\end{equation} 
the $(N^i_{\tau})_{\tau\in [-\log(T),\infty)}$ evolve by the so called rescaled mean curvature flow (RMCF):
\begin{equation}
\partial_{\tau} x =\vec{H}+\frac{x^{\perp}}{2}=: \langle \phi, \nu \rangle \nu.
\end{equation}
Under this rescaling (and time re-parametrization), the assumption of Theorem \ref{main_theorem} becomes
\begin{equation}\label{dist_ass_res}
    \lim_{\tau \rightarrow \infty }e^{k\tau}d_{H}\left(N^1_{\tau},N^2_{\tau}\right)=0,\qquad \forall k\in \mathbb{N}.
\end{equation}  

}
The argument in \cite{MartinHagemayerSesum} relies on analysis of the decay of $\phi$ (as $\phi\equiv 0$ precisely on  self-shrinkers), by means of analysing the quantity 
\begin{equation}\label{tilde_I_def}
\tilde{I}(\tau)=\int_{N^2_{\tau}} |\phi|^2d\tilde{\mu}_{\tau},
\end{equation}
and its logarithmic derivative (the so called ``Dirichlet--Einstein quotient''). Here
\begin{equation}\label{Gauss_measure_def}
d\tilde{\mu}_{\tau}=\exp\left(-\frac{|x|^2}{4}\right) d\mu_{\tau}.
\end{equation}
Note that while $\phi$ satisfies the nice linear PDE over $N^2_{\tau}$
\begin{equation}\label{phi_ev_and_L_def}
\partial_{\tau}\phi=L\phi ,\qquad \textrm{where}\qquad  L\phi: = \Delta \phi-\frac{1}{2}\langle x,\nabla \phi \rangle +\left(|A|^2+\frac{1}{2}\right)\phi,
\end{equation}
and while $\tilde{I}(\tau)$ gives a good measure for how non self-shrinking $N^2_{\tau}$ is, integrals involving $\phi$ are intrinsically inadequate to measuring differences between two non-shrinking flows (see also Remark \ref{why_not_lin_over_shrinker_both}). We thus need to take a different approach if we wish to allow $N^1_{\tau}$ to be non-shrinking.

\bigskip

Instead, as $N^i_{\tau}$, $i=1,2$, have the same compact singularity model, we may write $N^2_{\tau}$ as a normal graph of a function $u:N^1_{\tau}\to\mathbb{R}$ with a small norm, for $\tau$ sufficiently large. Similarly to $\phi$, we show that the height function $u$ satisfies a PDE over $N^1_{\tau}$ of the form
\begin{align}\label{Lin_with_err}
        \partial_{\tau} u = Lu+E(u),
\end{align}
where $L$ is as in \eqref{phi_ev_and_L_def} and $E$ is an error term, which by \eqref{dist_ass_res} can be shown to satisfy
\begin{equation}\label{err_est_intro}
|E(u)(x)| \leq e^{-\tau}(|u(x)|+|\nabla u(x)|).
\end{equation}

Our proof thus relies on analysing the smallness of $u$ rather than $\phi$. Instead of the quantity $\tilde{I}$ from \eqref{tilde_I_def} and its logarithmic derivative, we consider the energy given by the $L^2$ norm of the height,
    \begin{align*}
        I(\tau):=\int_{N_{\tau}^1} u^2 d\tilde{\mu}_t
    \end{align*} and set the parabolic frequency as
    \begin{align*}
        U(\tau):=\frac{\int_{N_{\tau}^1} uLu  d\tilde{\mu}_{\tau}}{\int_{N_{\tau}^1} u^2  d\tilde{\mu}_{\tau}}.
    \end{align*}

Note that $U$ is the Reyleigh quotient for the operator $L$. As is common in frequency based argument,  we show that $U$ is bounded below by a constant $U_\infty>-\infty$. Since $U$ is, up to integrable errors,  the logarithmic derivative of $I$, we can conclude $I$ cannot decay super-exponentially. The integrability of the coefficient in the right hand side of \eqref{err_est_intro} (and its square), as well as the $L^1$ in time integrability of the shrinker expression $\|\phi\|_{L^2(N^1_{\tau},d\tilde{\mu}_{\tau})}$ (owing to Schulze's Simon's Lojasiewicz inequality \cite{Schulze_cpt})  play a key role here (see Theorem \ref{L_with_errors} and Lemma \ref{non_shrinker_est}).    

Our frequency analysis of $u$ has similarities with the one performed by Bauldauf, Ho and Lee \cite{BHL} for solution of the heat equation with errors on self-shrinkers (which itself is based on Colding and Minicozzi's work \cite{CM21}). Note however that several aspects of our analysis are different: Apart from working in the rescaled picture from the get go, the flow we consider is not self-similar, the operator is different (even after scaling), and our frequency function is different. Importantly, in our setting, there is no term of the form $|T-t|^{-2\kappa}$ as in \cite{BHL}, so the sharp assumption \eqref{dist_ass} suffices.

\begin{remark}\label{why_not_lin_over_shrinker_both}
It is natural to try linearizing  both $N^1_{\tau}$ and $N^{2}_{\tau}$ over the static flow $\Sigma$, and consider the evolution of height difference over that fixed shrinker. Note however that in this case, the error term in \eqref{err_est_intro} would have a free $|\nabla^2u|$ term on the right hand side (with a time integrable coefficient). Such errors are not easily amenable for frequency analysis (see \cite{KS_simple,CM21,BHL}). A similar challenge is likely to appear if one tries to obtain an analogous result to Theorem \ref{main_theorem} for the Ricci flow, using the ideas of \cite{Ks}, or if one tries to use the two shrinker quantities on $N^1_{\tau}$ and $M^2_{\tau}$ to derive Theorem \ref{main_theorem} in a manner more similar to \cite{MartinHagemayerSesum}. 
\end{remark}

The organization of the paper is as follows: In Section \eqref{sec_err} we derive Theorem \ref{L_with_errors}, showing that solution $u$ to \eqref{Lin_with_err} with errors satisfying morally \eqref{err_est_intro}  over a RMCF converging to compact self shrinkers can not converge to zero super-exponentially. In Section \ref{main_thm_proof} we show how Theorem \ref{main_theorem} reduces to Theorem \ref{L_with_errors}. The Appendix derives the RMCF equation in normal graphical gauge over another RMCF, from which the estimate \eqref{Lin_with_err} and \eqref{err_est_intro} follow. 

\subsubsection*{Acknowledgements} We wish to thanks Brett Kotschwar and Brian White for useful discussions. This project has received funding from the European Research Council (ERC) under the European Union's Horizon 2020 research and innovation programme, grant agreement No. 101116390. JMDH was supported by the EPSRC through the grant EP/Y017862/1, \textit{Geometric Flows and the Dynamics of Phase Transitions}.


\section{Linearlized RMCF with errors}\label{sec_err}
The main theorem of this section is the following:
\begin{theorem}\label{L_with_errors}
Let $(N_t)_{\tau\in [0,\infty)}$ be a solution to RMCF such that there exists a compact smooth self shrinker $\Sigma$ for which $N_{\tau}\rightarrow \Sigma$ as $\tau \rightarrow \infty$, with multiplicity one. Let $u:N_{\tau}\rightarrow \mathbb{R}$ be a smooth function  satisfying 
\begin{equation}\label{err_est_ass}
|(\partial_{\tau}-L)u| \leq C(\tau)(|u|+|\nabla u|),
\end{equation}
where $L$ is given by \eqref{phi_ev_and_L_def} and  $C(\tau)\in L^2([0,\infty))\cap L^1([0,\infty))$. If for every $k\in \mathbb{N}$
\[
e^{k\tau}\|u(\tau)\|_{C^0(N_{\tau})}\rightarrow 0\qquad \textrm{as }\; \tau \rightarrow \infty,  
\]
then $u\equiv0$. 
\end{theorem}

In order to establish Theorem \ref{L_with_errors} for $N_{\tau}$ which is not self-shrinking, we need to  control some derivatives of geometric quantities along the flow, in a time integrable fashion. This is done in the following two Lemmas, the first of which is  a recasting of \cite[Lemma 3.1]{Schulze_cpt}.
\begin{lemma}\label{phi_int}
Let $N_t$ be as in Theorem \ref{L_with_errors} and denote by $varphi$ the shrinker quantity 
\begin{equation}
\phi:= H+\frac{1}{2}\langle x, \nu \rangle,
\end{equation}
where $\nu$ is the inner pointing normal. Then 
\begin{equation}
\int_{0}^\infty \|\phi\|_{L^2(M_t)} <\infty.
\end{equation}
Here the $L^2$ norm is the Gaussian $L^2$ norm.
\end{lemma}
\begin{proof}
Recall the $F$ functional 
\[
F[M]=\frac{1}{(4\pi)^{n/2}}\int_{M} e^{-|x|^2/4},
\]
has $\phi$ as its gradient, and so Simon's Lojasiewicz inequality \cite{Simon,Schulze_cpt}, and the convergence of $N_{\tau}$ to $\Sigma$ imply that there exists some $\theta\in (0,1)$ and $\tau_0$ such that for every $\tau \geq \tau_0$  
\[
(F[N_{\tau}]-F[\Sigma])^{1-\theta} \leq \left(\int_{N_{\tau}}|\phi|^2d\tilde{\mu}_{\tau}\right)^{1/2},
\]
where we recall that $d\tilde{\mu}_{\tau}$ is the Gaussian measure given by \eqref{Gauss_measure_def}.
Since rescaled MCF is the gradient flow of $F$, we have, for every $\tau \geq \tau_0$
\begin{align}
\frac{d}{d{\tau}}(F[N_{\tau}]-F[\Sigma]) &= -\left(\int_{N_{\tau}}|\phi|^2d\tilde{\mu}_{\tau}\right)^{1/2}\left(\int_{N_{\tau}}|\phi|^2d\tilde{\mu}_{\tau}\right)^{1/2}\nonumber\\
& \leq  -(F[N_{\tau}]-F[\Sigma])^{1-\theta} \left(\int_{N_{\tau}}|\phi|^2d\tilde{\mu}_{\tau}\right)^{1/2},
\end{align}
so 
\[
\frac{1}{\theta} \frac{d}{d\tau}\left(\left( F[N_{\tau}]-F[\Sigma] \right)^{\theta}\right) \leq  -\left(\int_{N_{\tau}}|\phi|^2d\tilde{\mu}_{\tau}\right)^{1/2}.
\]
Integrating this we obtain, for every $\tau_1\geq \tau_0$
\[
\int_{\tau_0}^{\tau_1} \left(\int_{N_{\tau}}|\phi|^2d{\tilde{\mu}}_{\tau}\right)^{1/2} \leq \frac{1}{\theta} (F[N_{\tau_0}]-F[\Sigma])^\theta. 
\]
The claim readily follows
\end{proof}

\begin{lemma}\label{non_shrinker_est}
Let $N_{\tau}$ be as in Theorem \ref{L_with_errors}. Then
\begin{equation}\label{higher_order_phi}
\int_{0}^\infty \|\phi\|_{C^2(N_{\tau})}d\tau < \infty. 
\end{equation}
In particular, there exists a continuous function $D:[0,\infty)\rightarrow \mathbb{R}_{+}$ with $\int_{0}^{\infty} D(\tau)d\tau<\infty$ such that
\begin{equation}
\|g'(\tau)\|_{g(\tau)}+ \left|\frac{d}{d\tau}|A|^2\right|+ {|\phi(\tau)|}+ \left|\frac{d}{d\tau}\phi\right| \leq D(\tau),
\end{equation} 
where $A$ is the second fundamental form of $N_{\tau}$ and $g(\tau)$ is the induced metric on $N_{\tau}$.
\end{lemma}
\begin{proof}

Note that as $\phi$ satisfies the parabolic equation 
\begin{equation}\label{phi_eq}
\partial_{\tau}\phi=L\phi,
\end{equation}
setting $f(\tau):=\|\phi\|^2_{L^2(N_{\tau},d\tilde{\mu}_{\tau})}$ we have 
\begin{equation}
f'(\tau) = 2\int_{N_{\tau}} \phi L\phi {- \frac{1}{2}\phi^4}d\tilde{\mu}_{\tau}  = -2\int_{N_{\tau}} |\nabla \phi|^2+2\int (|A|^2+\frac{1}{2}-\frac{1}{2}\phi^2)\phi^2d\tilde{\mu}_{\tau} \leq Cf(\tau),
\end{equation}
where $C=3(\sup_{\Sigma}|A|^2+1)$, and the last inequality holds for $\tau$ large enough by the convergence of $N_{\tau}$ to $\Sigma$.  Thus for every $\tau_2>\tau_1$ (provided $\tau_1$ is large enough)
\begin{equation}\label{bdd_growth}
f(\tau_2) \leq e^{C(\tau_2-\tau_1)}f(\tau_1).
\end{equation}

We now claim that \eqref{bdd_growth} and Lemma \ref{phi_int} imply that
\begin{equation}\label{sum_i}
\sum_{i=0}^{\infty}\sup_{\tau\in [i,i+2]}\|\phi\|_{L^2(N_{\tau})}<\infty.
\end{equation}
Indeed, if \eqref{sum_i} were not to hold, we could assume w.l.o.g. that 
\begin{equation}\label{sum_i4}
\sum_{i=0}^{\infty}\sup_{\tau\in [4i,4i+2]}\|\phi\|_{L^2(N_\tau)}=\infty.
\end{equation}
{
In light of \eqref{bdd_growth}, we obtain 
\begin{align*}
    e^{-4C}\sup_{\tau\in[4i,4i+2]} f(\tau)\leq \inf_{\tau \in[4i-2,4i]}f(\tau),
\end{align*}
for every $i$, from which we conclude
\begin{equation}
\sum_{i=1}^{\infty}\int_{4i-2}^{4i}\|\phi\|_{L^2(N_{\tau})} \geq e^{-2C}\sum_{i=0}^{\infty}\sup_{\tau\in [4i,4i+2]}\|\phi\|_{L^2(N_\tau)},
\end{equation}}

contradicting Lemma \ref{phi_int}.

\bigskip

Now \eqref{sum_i} clearly implies that, 
\begin{equation}\label{L^2_in_time_too}
\int_{0}^{\infty}\left(\int_{0}^{1}\int_{N_{\tau+\sigma}}\|\phi\|^2_{L^2(N_{\tau+\sigma})}\right)^{1/2} \leq \int_{0}^{\infty}\sup_{\sigma\in [\tau,\tau+1]} \|\phi\|_{L^2(N_\sigma)}<\infty. 
\end{equation} 
The fact that $\phi$ satisfies the parabolic equation \eqref{phi_eq}, whose corfficients converge to the coefficient of $L$ on the shrinker $\Sigma$, allows one to use standard parabolic PDE theory to obtain the existence of $C<\infty$ such that for every $\tau$ 
\[
\|\phi\|_{C^2\left(N_{\tau}\times \left[\tau+\frac{1}{2},\tau+1\right]\right)} \leq C\left(\int_{0}^{1}\int_{N_{\tau+\sigma}}\|\phi\|^2_{L^2(N_{\tau+\sigma})}\right)^{1/2}. 
\] 
Combined with \eqref{L^2_in_time_too}, this  yields  \eqref{higher_order_phi}.

\bigskip

Finally, note that standard derivation of the evolution equation \cite[Lemma 2.1, Lemma 2.4]{MartinHagemayerSesum} read
\begin{equation}
g'(\tau)=-2\phi A,\qquad \frac{d}{d\tau}|A|^2=2\langle A, \nabla^2 \phi \rangle +2\phi \mathrm{tr}(A^3),
\end{equation} 
which, combined with the already stated  \eqref{phi_eq}, \eqref{higher_order_phi} and the smooth convergence to the compact self shrinker, yields the moreover part of the lemma.

\end{proof}

\begin{proof}[Proof of Theorem \ref{L_with_errors}]
Setting 
\[
E(u):= \left(\partial_\tau-L\right)u,
\]
we have by assumption that
        \begin{align*}
        |E(u)|\leq C(\tau)\left(|u|+|\nabla u|\right).
        \end{align*}
        We recall the following definitions
        \begin{align*}
            I(\tau) &= \int_{N_{\tau}} u^2 d\tilde{\mu}_{\tau}\\
            U(\tau) &= \frac{2\int_{N_\tau} uLu d\tilde{\mu}_{\tau}}{I}= \frac{2\int_{N_\tau} uLu d\tilde{\mu}_{\tau}}{\int_{N_\tau} u^2 d\tilde{\mu}_\tau}=\frac{2\int_{N_\tau} -\norm{\nabla u}^2+\left(\norm{A}^2+\frac1{2}\right)u^2 \ d\tilde{\mu}_\tau}{I}.
        \end{align*}
Note that $U(\tau)$ is a Rayleigh quotient corresponding to the operator $L$. Thus, by the continuity of the first eigenvalue and as the $N_{\tau}$ converge to $\Sigma$, we can find some $\Lambda<\infty$ such that $U(\tau)<\Lambda$ for every $\tau$. Our main goal is to show that $U$ is bounded from below as well.  We calculate
        \begin{align}
            \ddt I(\tau) &= 2\int_{N_\tau} u (L u+E(u))-\phi^2 u^2d\tilde{\mu}_{\tau}\nonumber\\
            \ddt \log(I(\tau)) &= \frac{2\int_{N_\tau} u (L u+E(u))-\phi^2u^2{d}\tilde{\mu}_{\tau}}{\int_{N_\tau} u^2 d\tilde{\mu}_{\tau}}\label{log_der}\\
            &\geq  U(\tau)-\frac{2C(\tau)\int_{N_\tau} u(|u|+|\nabla u|) d\tilde{\mu}_{\tau}}{\int_{N_\tau} u^2 d\tilde{\mu}_{\tau}}-D(\tau) \nonumber\\
            &\geq U(\tau)-3(C(\tau)+D(\tau)) -C(\tau)\frac{\int_{N_\tau} |\nabla u|^2 d\tilde{\mu}_{\tau}}{\int_{N_\tau} u^2 d\tilde{\mu}_{\tau}}\nonumber\\
            &\geq \left( 1+\frac{C(\tau)}{2} \right) U(\tau)-(3+\cA)(C(\tau)+D(\tau)) \nonumber
        \end{align}
        where $\cA=2\max_{\Sigma}\{|A|^2+\frac1{2}\}$, 
where $D$ is the function from Lemma \ref{non_shrinker_est}. We can now  calculate
\begin{equation}\label{der_def}
U'(\tau)=V_{\textrm{main}}(\tau)+V_{\textrm{err}}(\tau),
\end{equation}
where 
\begin{equation}
V_{\textrm{main}}(\tau):=\frac{2\int_{N_\tau} \partial_\tau u Lu+u L\partial_\tau u d\tilde{\mu}_{\tau}}{I}-\frac{4\int_{N_{\tau}} \left( uLu+uE(u) \right)\int_{N_{\tau}} uLu d\tilde{\mu}_{\tau}}{I^2}
\end{equation}
is the term that would appear if the flow $N_{\tau}$ were the static self shrinker $\Sigma$, and
\begin{align*}
V_{\textrm{err}}(\tau)&=-\frac{2\int_{N_\tau} \left(-\norm{\nabla u}^2+\left(\norm{A}^2+\frac1{2}\right)u^2\right)\phi^2+2g'(\nabla u ,\nabla u)-\frac{d|A|^2}{dt}u^2 d\tilde{\mu}_\tau}{I}\\
&+\frac{2\int_{N_\tau}\phi^2u^2d\tilde{\mu}_\tau\int_{N_\tau}uLu d\tilde{\mu}_\tau}{I^2}. 
\end{align*}
is the error term coming from $N_{\tau}$ not being static. Note that by Lemma \ref{non_shrinker_est} we can safely estimate
\begin{equation}\label{err_est}
V_{\textrm{err}}(\tau) \geq c(\mathcal{A})D(\tau)(U(\tau)-\Lambda),
\end{equation}
for some constant $c$ depending on $\mathcal{A}$.
Using the self adjointness of $L$, we can also estimate
        \begin{align*}
           V_{\textrm{main}}(\tau) &=\frac{4\int_{N_\tau}  (Lu)^2+ Lu E(u) d\tilde{\mu}_\tau}{I}-\frac{4\int_{N_\tau} \left( uLu+uE(u) \right)\int_{N_\tau} uLu d\tilde{\mu}_\tau}{I^2}\\
           &=4\left(\frac{I\int_{N_\tau}  \left(Lu+\frac{E(u)}{2}\right)^2 -\left(\frac{E(u)}{2}\right)^2 d\tilde{\mu}_\tau}{I^2}\right)\\
           &-4\left(\frac{\int_{N_\tau} u\left( Lu+\frac{E(u)}{2} \right)+u\frac{E(u)}{2}d\tilde{\mu}_\tau\int_{N_\tau} u\left(Lu+\frac{E(u)}{2}\right)-u\frac{E(u)}{2} d\tilde{\mu}_\tau}{I^2}\right)\\
           &=4\left(\frac{I\int_{N_\tau}  \left(Lu+\frac{E(u)}{2}\right)^2 -\left(\frac{E(u)}{2}\right)^2 d\tilde{\mu}_\tau}{I^2}-\frac{\left(\int_{N_\tau} u\left( Lu+\frac{E(u)}{2} \right)d\tilde{\mu_\tau}\right)^2 -\left(\int_{N_t} u\frac{E(u)}{2}\right)^2 d\tilde{\mu}_\tau}{I^2}\right)\\
           &\geq-\frac{\int_{N_\tau} E(u)^2 d\tilde{\mu}_\tau}{I} \geq -2C^2(\tau)\frac{\int_{N_\tau} (|u|^2+|\nabla u|^2) d\tilde{\mu}_\tau}{I}\geq 2\left(U-(1+\cA)\right)C^2(\tau),
        \end{align*}        
where the first inequality follows from Cauchy-Schwartz and the second follow from \eqref{err_est_ass}. Combining this with \eqref{der_def} and \eqref{err_est} we obtain that there exists some constant $c=c(\mathcal{A})$ such that
\begin{equation}
(\Lambda-U)'(\tau) \leq c(\Lambda-U)(D(\tau)+C^2(\tau)),
\end{equation}
from which we deduce
            \begin{align*}
                \frac{d}{d\tau} \log(\Lambda-U)\leq c(C^2(\tau)+D(\tau)).
            \end{align*}

We integrate this formula over time to yield
        \begin{align*}
            \log\left(\frac{\Lambda-U(\tau)}{\Lambda-U(0)}\right)&\leq c\int_{0}^{\tau} (C^2(\sigma)+D(\sigma))d\sigma\\,
        \end{align*}
which after exponentiating both sides an rearranging yields
        \begin{align*}
            U(\tau)&\geq \Lambda + (U(0)-\Lambda)\exp\left(c\int_{0}^{\tau}(C^2(\sigma)+D(\sigma))d\sigma\right).
        \end{align*}
By the time integrability of $C^2$ and $D$, this implies that there exists some $U_{\infty}>-\infty$ such that
\begin{equation}
U(\tau)\geq U_{\infty},\qquad \tau \in [0,\infty).
\end{equation}
By \eqref{log_der} and as $C,D\in L^1([0,\infty)])$ we get that there exists some positive  $\lambda<\infty$ such  
\begin{equation}
\log(I)(\tau) \geq -\lambda \tau - \lambda.
\end{equation}

This contradicts the assumption that $u$ converges to zero faster than any exponential.
\end{proof}

\section{proof of the main theorem}\label{main_thm_proof}
In this section we will show how Theorem \ref{main_theorem} can be reduced to Theorem \ref{L_with_errors}
\begin{proof}[Proof of Theorem \ref{main_theorem}]
Consider the rescaled mean curvature flows $N^i_{\tau}$ obtained from $M^i_t$  by
\begin{equation}
N^i_{\tau}=e^{\tau/2}\left(M^{i}_{T-e^{-\tau}}-x_0\right).
\end{equation} 
By assumption there exist some compact self shrinkers $\Sigma^i$ ($i=1,2$) such that $N^i_\tau\rightarrow \Sigma^i$ in $C^{\infty}$ as $\tau \rightarrow \infty$. 
Note that the assumption \eqref{dist_ass} rescales to the fact that for every $k\in \mathbb{N}$
\begin{equation}\label{exp_conv_Hauss}
e^{k\tau}d_H(N^1_\tau,N^2_\tau)\rightarrow 0.
\end{equation}
In particular we have that $d_H(N^1_\tau,N^2_\tau)\rightarrow 0$, which in turn implies that \newline $\Sigma^1=\Sigma^2:=\Sigma$.

\bigskip

Since both $N^i_\tau$ converge smoothly to $\Sigma$, for all $\tau\geq \tau_0$ the hypersurface $N^2_\tau$ can be written as a normal graph of a smooth function $u$ over $N_\tau:=N^1_\tau$ satisfying $u\rightarrow 0$ in $C^{\infty}$ as $\tau \rightarrow \infty$. Observe that \eqref{exp_conv_Hauss} implies that for every $k\in \mathbb{N}$
\begin{equation}\label{u_c_0_conv}
e^{k\tau}\sup_{x\in N_\tau}|u(x,\tau)|\rightarrow 0.
\end{equation}
\bigskip

Now, using Lemma \ref{app_lin_RMCF}, the height function  $u$ of $N^2_{\tau}$ over $N_\tau=N^1_\tau$ satisfies a parabolic quasilinear  equation of the form 
\begin{equation} 
u_{\tau}=Lu+\varepsilon_1(x,u,\nabla u)\mathrm{div}(\mathcal{E} (\nabla u))+\varepsilon_2(x,u,\nabla u),
\end{equation}\label{par_eq_main}
where $Lu=\Delta u -\frac{1}{2}\langle x, \nabla u \rangle+\left(|A|^2+\frac{1}{2}\right)u$, and where $\mathcal{E}:TN_{\tau}\rightarrow TN_{\tau}$ and $\varepsilon^i:N_{\tau}\rightarrow \mathbb{R}$ satisfy the pointwise estimates
{\begin{equation}\label{E_est_main}
    \|\partial_\tau^j\nabla^i \mathcal{E}\|+\|\partial_\tau^j\nabla^i \varepsilon_1\| \leq C \sum_{l=0}^{1+i+2j} |\nabla^l u(x)|, \qquad \|\partial_\tau^j\nabla^i \varepsilon_2\| \leq C \sum_{l=0}^{1+i+2j} |\nabla^l u(x)|^2,
\end{equation} }   
for some universal $C$ and for every $i,j\geq 0 $ with $i+j\leq 1$, and for every $t\geq t_0$.

\bigskip

Since $u$ approaches $0$ in $C^{\infty}$, it follows from standard quasilinear parabolic theory that there exists some $C<\infty$ such that for every $\tau \geq \tau_0$  and $i=1,2$
\begin{equation}
\sup_{x\in N_\tau}|\nabla^i u(x,\tau)| \leq C\sup_{\sigma \in [\tau-1,\tau]}\sup_{x\in N_\sigma}|u(x,\sigma)|.
\end{equation}
Combined with \eqref{u_c_0_conv} (with $k=1$) this gives that
\begin{equation}
e^{\tau}\|u(x,\tau)\|_{C^2(N_\tau)}\rightarrow 0.
\end{equation}
By \eqref{par_eq_main} and \eqref{E_est_main}, applied with $j=0$ and $i=0,1$ we therefore get that there exists some $C<\infty$ such that 
\begin{equation}
|u_\tau-Lu| \leq Ce^{-\tau}(|u(x)|+|\nabla u(x)|).
\end{equation}
Thus, $u$ satisfies the assumption of Theorem \ref{L_with_errors}, and so $u=0$ for $\tau \geq \tau_0$. This implies that $N^1_\tau=N^2_\tau$ for $\tau\geq \tau_0$. 

\bigskip 

Rescaling back to the original MCF, we get that there exists some $\varepsilon>0$ such that $M^1_t=M^2_t$ for $t\in [T-\varepsilon,T)$. Since smooth MCF supports backwards uniqueness \cite{KS_simple,BU}, the Theorem follows.
\end{proof}

\appendix 

\section{The normal graphical gauge of rescaled  mean curvature flow over another rescaled mean curvature flow}\label{app}

We establish the following result, which is probably standard, providing the equation satisfied by the graph function of rescaled MCF over another rescaled mean curvature flow. This is the RMCF analogue of a similar result for minimal surface \cite[Lemma 2.26]{CM_minimal} and for MCF \cite[Lemma 4.19]{Her_reif}.

\begin{lemma}\label{app_lin_RMCF}
Let $(N_{\tau})_{\tau\in [\tau_0,\infty)}$  and $(N^2_\tau)_{\tau \in [\tau_0,\infty)}$ be two rescaled mean curvature flow such that $N_\tau$ is of bounded extrinsic geometry, and such that $N^2_\tau$ is a normal graph of a function $u$ over $N_\tau$ with small $C^5$ norm. Then
\begin{equation}\label{par_eq}
u_{\tau}=Lu+(1+\varepsilon_1(x,u,\nabla u))\mathrm{div}(\mathcal{E} (\nabla u))+\varepsilon_1(x,u,\nabla u)(\Delta u+\mathrm{div}(\mathcal{E} (\nabla u)))+\varepsilon_2(x,u,\nabla u),
\end{equation}
where $Lu=\Delta u -\frac{1}{2}\langle x, \nabla u \rangle+\left(|A|^2+\frac{1}{2}\right)u$, and where $\mathcal{E}:TN_{\tau}\rightarrow TN_{\tau}$ and $\varepsilon_i:N_{\tau}\rightarrow \mathbb{R}$ satisfy the pointwise estimate
\begin{equation}\label{E_est}
\|\partial_\tau^j\nabla^i \mathcal{E}\|+\|\partial_\tau^j\nabla^i \varepsilon_1\| \leq C \sum_{l=0}^{1+i+2j} |\nabla^l u(x)|, \qquad \|\partial_\tau^j\nabla^i \varepsilon_2\| \leq C \sum_{l=0}^{1+i+2j} |\nabla^l u(x)|^2
\end{equation}    
for some universal $C$ and for every $i,j\geq 0$ with $i+j\leq 1$.
\end{lemma}
\begin{proof}
Using the calculation and notations of \cite[Lemma 4.19]{Her_reif}, denoting by $\nu$ the normal to $N_{\tau}$, \cite[Eq. 4.29]{Her_reif} \footnote{and renaming $L$ therein by $\mathcal{E}$}  reads
\begin{equation}\label{exp_H}
\langle \vec{H}_2, \nu \rangle = (1+\varepsilon)\mathrm{div}((I+\mathcal{E})\nabla u)+H+u|A|^2+Q_{ij}A_{ij}, 
\end{equation}
where all the geometric quantities on the RHS are these of $N_\tau$. Additionally, by  \cite[Eq. 4.22]{Her_reif}, letting $\nu_2$  the normal to $N^2_{\tau}$
\begin{equation}
\nu_2=-u_iE_i+\nu+ Q_2(x,u(x),\nabla u(x)),
\end{equation}
where $Q_2$ is some analytic expression satisfying quadratic bounds in $u$ and $\nabla u(x)$.
Thus
\begin{equation}
\langle x+u(x)\nu(x), \nu_2(x) \rangle \langle \nu_2(x) ,\nu(x) \rangle = \langle x,\nu \rangle -\langle x, \nabla u \rangle+u+Q_2(x,u(x),\nabla u(x)),
\end{equation} 
where $Q_2$ was updated to a different expression satisfying the same qualitative quadratic bounds. Combined with \eqref{exp_H} we get that
\begin{align*}
&\left\langle \vec{H}_2+\frac{1}{2}\langle x+u(x)\nu(x), \nu_2(x) \rangle, \nu \right\rangle \\
&\qquad= H+\frac{1}{2}\langle x,\nu \rangle + \Delta u -\frac{1}{2}\langle x,\nu \rangle +\left(|A|^2+\frac{1}{2}\right)u \\
& \qquad+ (1+\varepsilon(x,u,\nabla u))\mathrm{div}(\mathcal{E} (\nabla u))+\varepsilon(x,u,\nabla u)(\Delta u+\mathrm{div}(\mathcal{E} (\nabla u)))\\
&\qquad+Q_2(x,u(x),\nabla u(x)).
\end{align*}
Note that the error terms appearing in the last two lines have the form that is asserted by the Lemma. Continuing as in \cite[Eq. 4.30-4.23]{Her_reif} yields the desired result. 
\end{proof}

\bibliography{epi}
\bibliographystyle{alpha}

\end{document}